\documentclass[amssymb,amstex,amsmath,11pt]{article}
\usepackage[margin=1.4in]{geometry}
\usepackage{amsmath,amssymb,latexsym,amsfonts,url}
\usepackage{graphicx}
\usepackage{amsthm}
\usepackage{tikz}
\usepackage{kotex}
\usepackage{pgfplots}
\usepackage{color}
\usepackage{epstopdf}
\usepackage{hyperref}
\allowdisplaybreaks
\usetikzlibrary{intersections, pgfplots.fillbetween}
\usepackage[all]{xy}
\pgfplotsset{compat=newest}
\pgfdeclarelayer{pre main}

\numberwithin{equation}{section}

\renewcommand{\epsilon}{\varepsilon}
\theoremstyle{plain}
\newtheorem{thm}{Theorem}[section]
\newtheorem*{thm*}{Theorem}
\newtheorem{prop}[thm]{Proposition}
\newtheorem{cor}[thm]{Corollary}
\newtheorem{lem}[thm]{Lemma}

\newtheorem*{lem*}{Lemma}

\newtheorem{rmk}[thm]{Remark}

\newtheorem*{theorem*}{Theorem}
\newtheorem*{proposition*}{Proposition}
\theoremstyle{definition}

\theoremstyle{remark}

 \usepackage{color}
 \usepackage[normalem]{ulem}

\newcommand{\lap}{\Delta}
\newcommand{\Om}{\Omega}
\newcommand{\Grad}{\overline{\nabla}}
\newcommand{\ddr}{\frac{\partial}{\partial r}}
\newcommand{\is}{\int_{\Sigma}}
\newcommand{\Vol}{{\rm Vol}}

\def\div{\textrm{div}}
\def\Ricm{\textrm{Ric}_M}
\def\Ricn{\textrm{Ric}_N}
\def\vol{\textrm{Vol}}

\def\Ric{\textrm{Ric}}

\begin{document}
\begin{title}
{Conformal structures and rigidity of complete stable minimal hypersurfaces}
\end{title}
\begin{author}{Jooyeon Park  \and  Keomkyo Seo}\end{author}

\date{\today}

\maketitle

\begin{abstract}
\noindent
We study complete stable minimal hypersurfaces in Riemannian manifolds under various curvature assumptions. In an $(n+1)$-dimensional complete oriented manifold with nonnegative scalar curvature, we prove that no complete oriented noncompact stable minimal hypersurface can be conformally equivalent to a bounded domain in an $n$-dimensional manifold with nonpositive scalar curvature. As a consequence, there exists no complete stable minimal hypersurface in $\mathbb{R}^{n+1}$ that is conformally equivalent to a bounded domain in $\mathbb{R}^n$. We also show that compact stable minimal hypersurfaces in manifolds with nonnegative scalar curvature have nonnegative smooth Yamabe invariant and we characterize the equality case. Next, we consider complete two-sided stable minimal hypersurfaces in ambient manifolds with pinched sectional curvature. Under an additional curvature condition that makes the second fundamental form a Codazzi tensor, we obtain a rigidity result under an $L^2$-condition on the second fundamental form and derive an upper bound for the first eigenvalue of the Laplacian. Finally, we establish that any complete noncompact two-sided minimal hypersurface immersed in a warped product manifold is stable, provided that the angle function is positive and the second derivative of the warping function is nonnegative.\\


\end{abstract}

\section{Introduction}

Stable minimal hypersurfaces form an important subclass of minimal submanifolds and have been widely studied for their geometric and analytic significance.
Recall that a minimal hypersurface $\Sigma$ in a Riemannian manifold $M$ is said to be stable if the second variation of its area is nonnegative under all compactly supported variations of $\Sigma$ in $M$.
More precisely, a minimal hypersurface $\Sigma$ in $M$ is called \textit{stable} if
$$\is |\nabla \varphi|^2 - ( |A|^2 + \Ricm(\eta,\eta) )\varphi^2 dv_g \geq 0 ~~ \text{for all} ~~ \varphi\in C^{\infty}_0(\Sigma),$$
where $\Ricm$ is the Ricci curvature of $M$, $\eta$ is the unit normal vector of $\Sigma$ in $M$, $A$ is the second fundamental form of $\Sigma$, and $dv_g$ is the volume element of $\Sigma$ with respect to the induced metric $g$.

Simple but important examples of stable minimal hypersurfaces are given by minimal graphs over $\mathbb{R}^n$. Bernstein's classical theorem \cite{B27} asserts that every minimal graph over $\mathbb{R}^2$ is planar. This rigidity result was subsequently extended to higher dimensions by Fleming \cite{F62}, De Giorgi \cite{G65}, Almgren \cite{A66}, and Simons \cite{S68}, who proved that every minimal graph over $\mathbb{R}^n$ is an affine plane whenever $n \leq 7$. On the other hand, Bombieri, De Giorgi, and Giusti \cite{BGG69} constructed explicit nonflat minimal graphs for $n \geq 8$, showing that such rigidity no longer holds in higher dimensions.

A closely related phenomenon arises in the stable Bernstein problem, which asks whether every complete oriented stable minimal hypersurface in $\mathbb{R}^{n+1}$ is necessarily a hyperplane. The examples constructed in \cite{BGG69} show that the answer is negative in high dimensions, while positive results are known in low dimensions (see \cite{CMR24, CL24, CLMS, CP79, FC&S80, M}). Under additional hypotheses, do Carmo and Peng \cite{CP80} and Shen and Zhu \cite{SZ98} proved rigidity results in higher dimensions by assuming that the second fundamental form has finite $L^2$-norm and finite $L^n$-norm, respectively.

In this paper, we study the geometry of complete two-sided stable minimal hypersurfaces and derive several related results. The paper is organized as follows. In Section~\ref{sec2}, we review basic preliminaries on the stability of minimal hypersurfaces. In particular, we recall both the characterization in terms of the first eigenvalue and the equivalent condition given by the existence of a positive supersolution. These results will be used throughout the paper.

In Section~\ref{sec3}, we investigate the conformal structure of complete stable minimal hypersurfaces. In their seminal work \cite{FC&S80}, Fischer-Colbrie and Schoen classified complete stable minimal surfaces in complete $3$-dimensional manifolds with nonnegative scalar curvature from the viewpoint of conformal geometry. More precisely, they proved the following:

\begin{thm}[\cite{FC&S80}] \label{FCS3}
    Let $M$ be a complete oriented $3$-dimensional Riemannian manifold of nonnegative scalar curvature and let $\Sigma$ be an oriented complete stable minimal surface in $M$.
    \begin{enumerate}
        \item[{\rm (i)}] If $\Sigma$ is compact, then $\Sigma$ is conformally equivalent to $\mathbb{S}^2$, or $\Sigma$ is a totally geodesic flat torus $\mathbb{T}^2$.
        \item[{\rm (ii)}] If $\Sigma$ is noncompact, then $\Sigma$ is conformally equivalent to the complex plane $\mathbb{C}$ or the cylinder.
    \end{enumerate}
\end{thm}

\noindent In particular, the Uniformization Theorem immediately implies that an oriented complete noncompact stable minimal surface $\Sigma$ in a complete oriented $3$-dimensional Riemannian manifold with nonnegative scalar curvature cannot be conformally equivalent to the open unit disk. We extend this observation to higher dimensions. More precisely, we prove the following.

\begin{thm*} [see Theorem \ref{thmnc}]
    Let $M$ be an $(n+1)$-dimensional complete oriented Riemannian manifold of nonnegative scalar curvature and let $(M_0,g_0)$ be an $n$-dimensional noncompact oriented manifold with nonpositive scalar curvature for $n\geq 3$.
    If $\Sigma$ is a complete oriented noncompact stable minimal hypersurface in $M$, then $\Sigma$ cannot be conformally equivalent to any bounded domain $\Om\subset M_0$.
\end{thm*}
\noindent As a special case, this result implies that there is no complete oriented stable minimal hypersurface in $\mathbb{R}^{n+1}$ conformally equivalent to a bounded domain in $\mathbb{R}^n$ (see Corollary~\ref{cnce}). On the other hand, the Yamabe constant of a conformal class, and the associated smooth Yamabe invariant obtained by taking the supremum over conformal classes, provide useful higher-dimensional scalar curvature invariants. In this spirit, we prove that any compact oriented stable minimal hypersurface $\Sigma$ in a Riemannian manifold with nonnegative scalar curvature has nonnegative smooth Yamabe invariant, and we characterize the zero case (see Theorem~\ref{thmc}).

In Section~\ref{sec4}, we consider complete two-sided stable minimal hypersurfaces whose second fundamental form $A$ satisfies an $L^2$-growth condition. In \cite{CP80}, do Carmo and Peng showed that every complete stable minimal hypersurface in Euclidean space with finite $L^2$-norm of $A$ is totally geodesic. We extend this argument to complete two-sided stable minimal hypersurfaces in complete Riemannian manifolds whose sectional curvatures are pinched, under the additional curvature condition $K_{n+1ijk}=0$ for $1\le i,j,k\le n$ with respect to an adapted local orthonormal frame. By the Codazzi equation this condition makes the second fundamental form a Codazzi tensor, which is precisely what is needed for the refined Kato inequality used below. The condition is automatic in space forms. Let $\Sigma$ be a hypersurface in an $(n+1)$-dimensional Riemannian manifold $M$. Assume that $M$ satisfies
\begin{align*}
K_{n+1ijk}=0 &\quad\text{for all } 1\leq i,j,k\leq n \quad\text{and} \\
K_1\leq K_M\leq K_2 &\quad\text{for some constants~~} K_1 ~\text{and}~ K_2,
\end{align*}
where $K_M$ denotes the sectional curvature $K_M$ of $M$ and that the covariant derivative of the curvature tensor satisfies
$$
|\nabla K|^2\le K_3^2|A|^2
$$
for some constant $K_3\ge 0$. Define the constant $\Gamma_\Sigma$ by
$$
\Gamma_\Sigma
:=
\sup_\Sigma \bigl(-(n^2+5n)K_1+2nK_2+4K_3+|A|^2+S_\Sigma\bigr),
$$
where $S_\Sigma$ denotes the scalar curvature of $\Sigma$. If $\Gamma_\Sigma\le 0$, then we prove that $\Sigma$ is totally geodesic under the above curvature assumptions and the $L^2$-growth condition (see Theorem~\ref{thmftc1}).  Under the same growth assumption, the first eigenvalue $\lambda_1(\Sigma)$ of the Laplacian on a complete stable minimal hypersurface $\Sigma$ in hyperbolic space satisfies
\begin{equation} \label{b1ev}
\frac{(n-1)^2}{4} \leq \lambda_1(\Sigma) \leq n^2.
\end{equation}
The lower bound was proved by Cheung and Leung \cite{C&L01}, while the upper bound was obtained in \cite{Seo11}. Subsequently, Dung and the second author \cite{D&S12} generalized this estimate to the setting of ambient manifolds with pinched negative sectional curvature. If $\Gamma_\Sigma > 0$, we further generalize the upper bound in \eqref{b1ev} to stable minimal hypersurfaces in complete Riemannian manifolds whose sectional curvatures are bounded between two constants, allowing the case of nonnegative sectional curvature (see Theorem~\ref{thmftc2}).

Meanwhile, Aledo and Rubio \cite{A&R16} investigated complete noncompact two-sided minimal surfaces in a class of $3$-dimensional warped product manifolds.
In particular, they proved that any complete noncompact two-sided minimal surface immersed in a $3$-dimensional warped product manifold is stable, provided that the angle function $\nu = g(\partial/\partial r,\eta)$ is positive and the warping function is positive with nonnegative second derivative (see Section~\ref{sec5} for details).
This condition on the angle function is satisfied, for example, when the surface is locally represented as a graph. Motivated by the argument in \cite{A&R16}, Section~\ref{sec5} is devoted to a higher-dimensional extension of this result. More precisely, we prove the following.

\begin{thm*}[see Theorem~\ref{thmw}]
Let $M=(0,\bar{r})\times_h N$ be an $(n+1)$-dimensional  warped product manifold with the metric $g=dr^2+h(r)^2g_N$, where $n\ge 2$ and $h(r)>0$ on $(0,\bar{r})$. Assume that $N$ is an $n$-dimensional Einstein manifold with $\mathrm{Ric}_N = (n-1)k g_N$ for some constant $k$. Let $\Sigma$ be a complete noncompact two-sided minimal hypersurface immersed in $M$ with $\nu\geq0$, where $\nu=g\!\left(\frac{\partial}{\partial r},\eta\right)$ is the angle function. If $h''(r)\geq0$ on $\Sigma$, then either $\nu>0$ everywhere on $\Sigma$ or $\nu\equiv 0$ on $\Sigma$. Moreover, if $\nu>0$ on $\Sigma$, then $\Sigma$ is stable.
\end{thm*}

\noindent As an application, we prove that if $\Sigma$ is a complete noncompact two-sided minimal hypersurface immersed in $\mathbb{R}^{n+1}$ or $\mathbb{H}^{n+1}$ and $p\notin\Sigma$ is fixed, then the corresponding angle function $\nu_p$ satisfies the same dichotomy. Moreover, $\nu_p>0$ implies that $\Sigma$ is stable (see Corollary~\ref{wceh}).

\section{Preliminaries} \label{sec2}

Let $(\Sigma,g)$ be an $n$-dimensional complete noncompact Riemannian manifold with $n\geq 2$, and let $q:\Sigma\to\mathbb{R}$ be a smooth function. Throughout this paper, we denote by $L_q$ the Schr\"odinger operator defined by
$$
L_q := \Delta_g - q,
$$
where $\Delta_g$ is the Laplace operator on $(\Sigma,g)$. For any bounded domain $D\subset\Sigma$, the first eigenvalue of $L_q$ of the operator $L_q$ on $D$ is variationally characterized by
$$\lambda_1^{L_q}(D) = \inf\left\{\int_D|\nabla \phi|^2 + q \phi^2\,dv_g~|~\phi\in C^{\infty}_{0}(D),\,\int_D\phi^2\,dv_g=1\right\}.$$
We need the following useful criterion, due to Fischer-Colbrie and Schoen \cite{FC&S80}, which characterizes the nonnegativity of the first eigenvalue of a Schr\"odinger operator by the existence of a positive supersolution.

\begin{prop}[\cite{FC&S80}]\label{prop0}
The following statements are equivalent:
\begin{enumerate}
    \item[{\rm (i)}] $\lambda_1^{L_q}(D)\geq0$ for every bounded domain $D\subset\Sigma$;
    \item[{\rm (ii)}] there exists a positive solution $f$ that satisfies $L_q f\leq0$ on $\Sigma$.
\end{enumerate}
\end{prop}

\begin{proof}
The implication ${\rm (i)}\Rightarrow {\rm (ii)}$ follows directly from Theorem 1 of Fischer-Colbrie and Schoen \cite{FC&S80}. Thus, it remains to prove ${\rm (ii)}\Rightarrow {\rm (i)}$. To see this, assume that there exists a positive function $f$ on $\Sigma$ satisfying $L_q f\leq 0$. Let $D\subset \Sigma$ be any bounded domain, and let $\phi\in C_0^\infty(D)$. Put $\phi=\varphi f$, where $\varphi$ is a compactly supported function on $D$. Then using
    $$\int_{D}\phi\lap \phi \,dv_g = -\int_{D}|\nabla \phi|^2 dv_g,$$
    we have
    $$\int_{D}|\nabla \phi|^2dv_g = -\int_{D}\varphi f\lap(\varphi f)\,dv_g =-\int_{D}(\varphi^2f\lap f + \varphi f^2\lap\varphi+2\varphi f\langle\nabla\varphi,\nabla f\rangle)\,dv_g.$$
   Thus the divergence theorem yields
    \begin{align}\label{i}
       \int_D (|\nabla \phi|^2 + q \phi^2) dv_g& = -\int_{D}[\varphi^2 f (\lap f - q f) + \varphi f^2\lap\varphi+2\varphi f\langle\nabla\varphi,\nabla f\rangle] \,dv_g\nonumber\\
        & = -\int_{D}[\varphi^2 f \, L_q f + \varphi f^2\lap\varphi+2\varphi f\langle\nabla\varphi,\nabla f\rangle]\,dv_g \nonumber\\
        & \geq -\int_{D}[ \varphi f^2\lap\varphi + \frac{1}{2}\langle\nabla(\varphi^2),\nabla (f^2)\rangle]\,dv_g \nonumber\\
        & =\int_{D} f^2|\nabla\varphi|^2 dv_g \geq 0,
    \end{align}
    which completes the proof.
\end{proof}

\noindent Let
$$
q=-|A|^2-\Ric_M(\eta,\eta)
$$
on $\Sigma$, where $A$ denotes the second fundamental form of $\Sigma$, $\Ric_M$ the Ricci curvature of $M$, and $\eta$ a unit normal vector field along $\Sigma$. Then
$$
L_q=L,
$$
where
$$
L=\Delta+|A|^2+\Ric_M(\eta,\eta)
$$
is the Jacobi operator. Thus, Proposition~\ref{prop0} immediately yields the following corollary.

\begin{cor}[\cite{FC&S80}]\label{fc&s}
Let $\Sigma$ be an $n$-dimensional complete noncompact minimal hypersurface immersed in an $(n+1)$-dimensional Riemannian manifold $M$, where $n\geq 2$. Then $\Sigma$ is stable if and only if there exists a positive smooth function $f$ on $\Sigma$ such that
$$
Lf \leq 0,
$$
where $L = \lap + |A|^2 + \mathrm{Ric}_M(\eta,\eta)$ is the Jacobi operator.
\end{cor}

\section{Conformal structures of complete stable minimal hypersurfaces in complete manifolds with nonnegative scalar curvature} \label{sec3}

Let $(\Omega,g_0)$ be an $n(\ge 3)$-dimensional Riemannian manifold. Suppose that $g$ is a complete metric on $\Omega$ conformal to $g_0$. Then there exists a positive smooth function $u$ on $\Omega$ such that
$$
g=u^{\frac{4}{n-2}}g_0.
$$
It is well known that the scalar curvatures $S_0$ and $S$ of $g_0$ and $g$, respectively, are related by
\begin{equation}\label{ye}
-\frac{4(n-1)}{n-2}\,\Delta_g(u^{-1})+S\,u^{-1}
=
S_0\,u^{-\frac{n+2}{n-2}},
\end{equation}
where $\Delta_g$ denotes the Laplace--Beltrami operator with respect to the metric $g$. On the other hand, let $M$ be the unit disk in the complex plane $\mathbb{C}$ equipped with the metric $ds^2=\lambda(z)\,|dz|^2$. In \cite{FC&S80}, Fischer-Colbrie and Schoen proved that if $ds^2$ is complete on $M$, then for every constant $a\ge 1$ there exists no positive solution $f$ of
$$
\Delta f-aKf=0
$$
on $M$, where $\Delta$ and $K$ denote the Laplacian and the Gaussian curvature of $M$, respectively. Later, Kato and Nayatani \cite{KN93} extended this result to higher-dimensional compact Riemannian manifolds with nonpositive scalar curvature. We show that the same conclusion still holds in the noncompact setting as follows. Throughout this section, by a bounded domain we mean a relatively compact domain.

\begin{prop} \label{prop1}
Let $(M_0,g_0)$ be an $n$-dimensional noncompact Riemannian manifold with $n\ge 3$ and nonpositive scalar curvature $S_0$ and let $\Omega$ be a bounded domain in $M_0$. Assume that $g$ is a complete metric on $\Omega$ conformal to $g_0$. Denote by $S$ the scalar curvature of the metric $g$. Then, for every constant $\beta \ge \frac{n}{4(n-1)},$
there is no positive function $f$ on $\Omega$ such that
$$
L_{\beta S}f \le 0
$$
on $\Omega$.
\end{prop}

\begin{proof}
Define
$$\beta_n=\frac{n}{4(n-1)}.$$
Consider any bounded domain $D$ in $\Om$ with the complete metric $g$. Then
$$L_{\beta_n S} = \Delta_g - \beta_n S.$$
For $\beta\geq\beta_n$, observe that
\begin{align*}
\int_{D} (|\nabla \phi|^2+\beta_n S \phi^2)\, dv_g & \geq \frac{\beta_n}{\beta}\int_{D} (|\nabla \phi|^2+\beta S \phi^2)\, dv_g\\
& \geq \frac{\beta_n}{\beta}\lambda_1^{L_{\beta S}}(D)
\end{align*}
for any smooth function $\phi$ with compact support in $D$ satisfying $\int_D\phi^2\,dv_g=1$. This gives
\begin{align}
\lambda_1^{L_{\beta_n S}}(D) \geq \frac{\beta_n}{\beta}\lambda_1^{L_{\beta S}}(D) \label{ineq: Proposition 3.1}
\end{align}
for $\beta\geq\beta_n$. We claim that
$$
\lambda_1^{L_{\beta_n S}}(D)<0
$$
for some bounded domain $D\subset (\Omega,g)$. Once this is established, it follows from \eqref{ineq: Proposition 3.1} and Proposition~\ref{prop0} that there exists no positive solution of
$$
L_{\beta S}f\le 0
$$
for any $\beta\ge \beta_n$. To prove the claim, by the conformal relation between $g$ and $g_0$, we write the metric $g$ in the form
$$
g=u^{\frac{4}{n-2}}g_0
$$
for some positive smooth function $u$ on $\Omega$. Since $S_0$ is nonpositive, it follows from \eqref{ye} that
$$
-\frac{4(n-1)}{n-2}\,\Delta_g(u^{-1})+Su^{-1}
=
S_0u^{-\frac{n+2}{n-2}}
\le 0.
$$
Therefore we have
$$
\Delta_g(u^{-1})\ge \frac{n-2}{4(n-1)}\,Su^{-1}.
$$
Define
$$
h=u^{-\frac{n}{n-2}}.
$$
Then a straightforward computation shows that
\begin{equation}\label{2.2}
\Delta_g h\ge \frac{2}{n}\frac{|\nabla h|^2}{h}+\beta_n Sh.
\end{equation}
    For a smooth function $\varphi$ on $\Om$ with compact support in $D$, we have
      \begin{align*}
        \lambda_1^{L_{\beta_nS}}(D)\int_{D} (\varphi h)^2 dv_g
        & \leq\int_{D} (-\varphi h \lap_g(\varphi h)+\beta_n S (\varphi h)^2)\, dv_g\\
        & = \int_{D} (-\varphi h^2 \lap_g \varphi -2\varphi h \,g(\nabla \varphi, \nabla h) -\varphi^2 h \lap_g h +\beta_n S (\varphi h)^2)\, dv_g \\
        & \leq \int_{D} \left(-\varphi h^2 \lap_g \varphi -\frac{1}{2} g\left(\nabla (\varphi^2), \nabla (h^2)\right) - \frac{2}{n}\varphi^2|\nabla h|^2\right) dv_g
    \end{align*}
    by the definition of $\lambda_1^{L_{\beta_n S}}(D)$ and (\ref{2.2}).
    Integration by parts gives
    \begin{equation}\label{2.3}
        \lambda_1^{L_{\beta_nS}}(D)\int_{D} (\varphi h)^2 dv_g \leq \int_{D} \left( h^2|\nabla \varphi|^2 -\frac{2}{n}\varphi^2|\nabla h|^2\right) dv_g.
    \end{equation}
    Choose a Lipschitz function $\varphi$ with a compact support satisfying
    $$0\leq \varphi\leq 1,\quad \varphi\equiv 1~ \text{on}~ B(R),\quad \varphi\equiv 0~\text{on}~\Om \backslash B(2R), \quad \text{and} \quad|\nabla \varphi|\leq \frac{2}{R},$$
    where $B_p(R)$ is a $g$-geodesic ball  of radius $R$ centered at a fixed point $p$.
    Putting this into (\ref{2.3}), we obtain
      \begin{equation}\label{2.4}
        \lambda_1^{L_{\beta_nS}}(B_p(2R))\int_{B_p(2R)} \varphi^2 h^2 dv_g \leq \frac{4}{R^2}\int_{B_p(2R)} h^2 dv_g -\frac{2}{n} \int_{B_p(R)}|\nabla h|^2dv_g.
    \end{equation}
Recall that $dv_g$ and $dv_{g_0}$ denote the volume elements of $g$ and $g_0$, respectively. Since
$$
\operatorname{vol}_{g_0}(\Omega)
=
\int_\Omega dv_{g_0}
=
\int_\Omega u^{-\frac{2n}{n-2}}\,dv_g
=
\int_\Omega h^2\,dv_g
\ge
\int_{B_p(2R)} h^2\,dv_g,
$$
and $\Omega$ is relatively compact in $(M_0,g_0)$ and hence has finite volume, the first term on the right-hand side of \eqref{2.4} tends to zero as $R\to\infty$. On the other hand, since $g$ is complete on the bounded domain $\Omega$, the conformal factor $u$ cannot be constant. Hence $h$ is also nonconstant, which shows that $|\nabla h|$ does not vanish identically. Thus, for sufficiently large $R$, the second term on the right-hand side of \eqref{2.4} is strictly negative. Therefore we conclude that
$$
\lambda_1^{L_{\beta_n S}}(B_p(2R))<0,
$$
which completes the proof in view of Proposition~\ref{prop0}.
\end{proof}

\noindent Combining Propositions~\ref{prop0} and \ref{prop1}, we obtain the following additional nonexistence result.

\begin{cor}\label{cor2}
Let $(M_0,g_0)$ be an $n(\ge 3)$-dimensional noncompact Riemannian manifold with nonpositive scalar curvature and let $\Omega\subset M_0$ be a bounded domain. Assume that $g$ is a complete metric on $\Omega$ conformal to $g_0$.  Let $P$ be a nonnegative function on $\Omega$. Then, for every
$\beta\ge \frac{n}{4(n-1)},$
there exists no positive solution $f$ on $\Omega$ of
$$
\Delta_g f-\beta S f+Pf\le 0,
$$
where $S$ is the scalar curvature of $(\Omega,g)$ and $\Delta_g$ denotes the Laplace-Beltrami operator with respect to $g$.
\end{cor}

\begin{proof}
Suppose, for contradiction, that there exists a positive solution $f$ on $\Omega$ satisfying
$$
\Delta_g f-\beta S f+Pf\le 0.
$$
Set
$$
q=\beta S-P,
$$
and let $D$ be any bounded domain in $(\Omega,g)$. Then Proposition~\ref{prop0} implies that
$$
\lambda_1^{L_q}(D)\ge 0.
$$
Hence, for every smooth function $\varphi$ with compact support in $D$ satisfying $\int_D \varphi^2\,dv_g=1,$
we have
$$
\int_D \bigl(|\nabla \varphi|^2+q\varphi^2\bigr)\,dv_g\ge 0.
$$
Since $q=\beta S-P$, it follows that
$$
\int_D \bigl(|\nabla \varphi|^2+\beta S\varphi^2\bigr)\,dv_g
\ge
\int_D P\varphi^2\,dv_g
\ge 0.
$$
Therefore we get
$$
\lambda_1^{L_{\beta S}}(D)\ge 0.
$$
Because $D$ was arbitrary, Proposition~\ref{prop0} applied again to the operator $L_{\beta S}$ yields the existence of a positive solution of
$$
L_{\beta S}f\le 0
$$
on $\Omega$. This contradicts Proposition~\ref{prop1}, which completes the proof.
\end{proof}

We are now in a position to prove Theorem~\ref{thmnc}.

\begin{thm}\label{thmnc}
Let $M$ be an $(n+1)$-dimensional complete oriented Riemannian manifold with nonnegative scalar curvature $S_M$ and let $(M_0,g_0)$ be an $n$-dimensional noncompact oriented Riemannian manifold with nonpositive scalar curvature for $n\ge 3$. Let $\Sigma$ be a complete oriented noncompact stable minimal hypersurface in $M$. Then $\Sigma$ cannot be conformally equivalent to any bounded domain $\Omega\subset M_0$.
\end{thm}

\begin{proof}
    Suppose, for contradiction, that $\Sigma$ is conformally equivalent to some bounded domain $\Omega\subset M_0$. Then $\Sigma$ may be identified with the domain $\Omega$ endowed with a complete metric
    $$g=u^{\frac{4}{n-2}}g_0,$$
    where $u$ is a positive smooth function on $\Omega$. Denote by $S$ the scalar curvature of $(\Omega,g)$. By the Gauss equation for a minimal hypersurface, we have
    \begin{equation*}
        \Ric_M(\eta,\eta)=\frac{S_M-S-|A|^2}{2},
    \end{equation*}
    where $\Ric_M$ is the Ricci curvature of $M$, $\eta$ is a unit normal vector field along $\Sigma$, and $A$ is the second fundamental form of $\Sigma$ in $M$. Therefore, the stability inequality gives
    \begin{align*}
        0
        &\le \int_\Omega \left(|\nabla \varphi|^2-\bigl(|A|^2+\Ric_M(\eta,\eta)\bigr)\varphi^2\right)\,dv_g \\
        &= \int_\Omega \left(|\nabla \varphi|^2+\left(\frac{S}{2}-\frac{S_M+|A|^2}{2}\right)\varphi^2\right)\,dv_g
    \end{align*}
    for every smooth function $\varphi$ with compact support in $\Omega$. It follows from Proposition~\ref{prop0} that there exists a positive solution $f$ on $\Omega$ satisfying
    $$\Delta_g f-\frac{S}{2}f+\frac{S_M+|A|^2}{2}f\le 0.$$
    Now observe that
    $$\frac{1}{2}\ge \frac{n}{4(n-1)}
    \qquad \text{and} \qquad
    \frac{S_M+|A|^2}{2}\ge 0.$$
    Hence Corollary~\ref{cor2} applies with
    $$\beta=\frac{1}{2} \qquad \text{and} \qquad P=\frac{S_M+|A|^2}{2},$$
    which is a contradiction. This completes the proof.
\end{proof}

\noindent As an immediate consequence of Theorem~\ref{thmnc}, we obtain the following corollary by taking $M=\mathbb{R}^{n+1}$ and $M_0=\mathbb{R}^n$, both equipped with their standard flat metrics.

\begin{cor}\label{cnce}
For $n\ge 3$, no complete oriented stable minimal hypersurface in $\mathbb{R}^{n+1}$ is conformally equivalent to any bounded domain in $\mathbb{R}^n$.
\end{cor}

\medskip
We now consider the higher-dimensional case in the compact setting. Let $(\Sigma, g)$ be an $n$-dimensional compact Riemannian manifold for $n\geq 3$. The \textit{Yamabe constant} of   the conformal class $[g]$  is defined by
$$
Y(\Sigma,[g]) :=\inf_{u>0} \frac{\int_{\Sigma} \left(\frac{4(n-1)}{n-2}|\nabla u|^2+S_{\Sigma} u^2\right)dv_{g}}{\left(\int_{\Sigma} u^{\frac{2n}{n-2}}dv_{g}\right)^{\frac{n-2}{n}}}.
$$
Equivalently, if $\tilde g=u^{\frac{4}{n-2}}g$, this quotient is the normalized total scalar curvature of $\tilde g$. We then define the \textit{smooth Yamabe invariant}, also called the \textit{$\sigma$-invariant}, by
$$
\sigma(\Sigma):=\sup_{[g]}Y(\Sigma,[g]),
$$
where the supremum is taken over all smooth Riemannian metrics on $\Sigma$.  By definition,  $Y(\Sigma,[g])$ depends only on the conformal class $[g]$. 
When $n=2$, the Yamabe functional reduces to
$$
E(g)=\int_{\Sigma} S\,dv_g=\int_{\Sigma} 2K\,dv_g,
$$
where $S$ denotes the scalar curvature of $\Sigma$ and $K$ its Gaussian curvature. In this case,
$$
\sigma(\Sigma)=4\pi\chi(\Sigma),
$$
where $\chi(\Sigma)$ is the Euler characteristic of $\Sigma$. Thus the Yamabe invariant may be viewed as a higher-dimensional analogue of the Euler characteristic.

\begin{lem}\label{y}
Let $M^{n+1}$ be an oriented  complete Riemannian manifold with nonnegative scalar curvature for $n\ge 3$. If $\Sigma^n$ is an oriented compact stable minimal hypersurface in $M$ with $\sigma(\Sigma)\le 0$, then $\Sigma$ is totally geodesic and Ricci curvature ${\rm Ric_M}$ of $M$ satisfies ${\rm Ric_M}(\eta,\eta)=0$ along $\Sigma$, where $\eta$ is a unit normal vector field along $\Sigma$. Moreover, the scalar curvature of $M$ vanishes on $\Sigma$.
\end{lem}

\begin{proof}
Let $g$ be the induced metric on $\Sigma$. Since $\Sigma$ is minimal, the Gauss equation gives
$$
S_\Sigma=S_M-2\Ric_M(\eta,\eta)-|A|^2,
$$
where $S_\Sigma$ and $S_M$ denote the scalar curvatures of $\Sigma$ and $M$, respectively, and $A$ is the second fundamental form of $\Sigma$. Thus
$$
\Ric_M(\eta,\eta)+|A|^2=\frac{1}{2}\bigl(S_M-S_\Sigma+|A|^2\bigr).
$$
Since $\Sigma$ is stable, for every $f\in C^\infty(\Sigma)$ we have
\begin{align*}
0
&\le \int_\Sigma \Bigl(|\nabla f|^2-(\Ric_M(\eta,\eta)+|A|^2)f^2\Bigr)\,dv_g \\
&= \int_\Sigma \left(|\nabla f|^2+\frac{1}{2}(S_\Sigma-S_M-|A|^2)f^2\right)\,dv_g.
\end{align*}
Because $S_M\ge 0$, it follows that
$$
\int_\Sigma \left(\frac{4(n-1)}{n-2}|\nabla f|^2+S_\Sigma f^2\right)\,dv_g\ge 0
$$
for all $f\in C^\infty(\Sigma)$. Hence 
$$
Y(\Sigma, [g])\ge 0.
$$
Moreover, since $\sigma(\Sigma)\le 0$, we see that
$$
Y(\Sigma, [g])=0.
$$
By the solution of the Yamabe problem, there exists a positive smooth function $u$ on $\Sigma$ such that the conformal metric
$$
\tilde g=u^{\frac{4}{n-2}}g
$$
has scalar curvature $\tilde S_\Sigma\equiv 0$. Equivalently,
$$
-\frac{4(n-1)}{n-2}\Delta u+S_\Sigma u=0
$$
on $\Sigma$. Multiplying by $u$ and integrating over $\Sigma$, we obtain
\begin{align}
\int_\Sigma \left(\frac{4(n-1)}{n-2}|\nabla u|^2+S_\Sigma u^2\right)\,dv_g=0. \label{id:01}
\end{align}
Applying the stability inequality with $f=u$, we also have
\begin{align}
0\le \int_\Sigma \left(2|\nabla u|^2+S_\Sigma u^2\right)\,dv_g. \label{ineq:02}
\end{align}
Since $\frac{4(n-1)}{n-2}>2$, from (\ref{id:01}) and (\ref{ineq:02}), it follows that $\nabla u=0$, which implies that $u$ is constant. Hence we see that
$$
S_\Sigma\equiv 0
$$
on $\Sigma$. Using the stability inequality again together with $S_\Sigma\equiv 0$, we obtain
$$
\int_\Sigma (S_M+|A|^2)\,dv_g\le 0.
$$
Since $S_M\ge 0$, it follows that
$$
S_M\equiv 0
\qquad\text{and}\qquad
|A|^2\equiv 0
$$
along $\Sigma$. In particular, $\Sigma$ is totally geodesic. The Gauss equation then reduces to
$$
0=S_\Sigma=S_M-2\Ric_M(\eta,\eta)-|A|^2,
$$
and therefore
$$
\Ric_M(\eta,\eta)=0
$$
along $\Sigma$.
\end{proof}

The second author \cite[Theorem 1.4]{Seo12} proved that if $M^{n+1}$ is compact with scalar curvature $S_M\ge k>0$ and $\Sigma\subset M$ is a compact stable minimal hypersurface, then  $\sigma(\Sigma)>0$ for $n\ge3$ and $\int_\Sigma S_\Sigma d\mu \ge k\operatorname{Vol}(\Sigma)$. In the following theorem, we allow the ambient scalar curvature to have a nonnegative lower bound, including the borderline case $k=0$, and establish rigidity in the equality cases.

\begin{thm}  \label{thmc}
Let $M$ be an oriented complete Riemannian manifold of dimension $n+1$ for $n\ge 3$ and assume that its scalar curvature satisfies $S_M\ge k$ for some constant $k\ge 0$. Let $\Sigma\subset M$ be an oriented compact stable minimal hypersurface without boundary. Denote by $S_\Sigma$ the scalar curvature of $\Sigma$. Then we have the following.
    \begin{itemize}
        \item[{\rm (i)}] $\sigma(\Sigma)\geq0$. Moreover, if equality holds, then $\Sigma$ is totally geodesic and $S_\Sigma\equiv 0$ on $\Sigma$.
        \item[{\rm (ii)}] $\is S_{\Sigma} \geq k \mathrm{Vol} (\Sigma)$. Moreover, if equality holds, then $\Sigma$ is totally geodesic.
    \end{itemize}
\end{thm}

\begin{proof}
    By the stability of $\Sigma$, we have
    \begin{align}\label{2.7}
        0 & \leq \is 2|\nabla \varphi|^2+(S_{\Sigma}-S_M-|A|^2)\varphi^2 \nonumber\\
          & \leq \is \frac{4(n-1)}{n-2}|\nabla \varphi|^2 + S_{\Sigma} \varphi^2 - (S_M + |A|^2)\varphi^2
    \end{align}
    for any smooth function $\varphi$ defined on $\Sigma$.
    In (\ref{2.7}), we have used  the fact that $2<\frac{4(n-1)}{n-2}$ for $n\geq3$. Then the assumption  that  $S_M \geq k\geq 0$ gives
    \begin{equation*}
        0 \leq \is (S_M+|A|^2)\varphi^2
          \leq \is \frac{4(n-1)}{n-2}|\nabla \varphi|^2+S_{\Sigma} \varphi^2,
    \end{equation*}
    which implies
    \begin{equation*}
        0 \leq \frac{\is \frac{4(n-1)}{n-2}|\nabla \varphi|^2 + S_{\Sigma}\varphi^2}{\left(\is \varphi^{\frac{2n}{n-2}}\right)^{\frac{n-2}{n}}},
    \end{equation*}
which shows that $\sigma(\Sigma)\geq0$.

We now consider the equality case. Suppose that $\sigma(\Sigma)=0$. Then, by Lemma \ref{y}, $\Sigma$ is totally geodesic, $\Ricm(\eta,\eta)=0$ and $S_M=0$ along $\Sigma$. Applying the Gauss formula, we have
$$
\Ricm(\eta,\eta)=\frac{1}{2}\bigl(S_M-S_\Sigma-|A|^2\bigr).
$$
    It then follows that $S_{\Sigma}=0$ on $\Sigma$.

    For {\rm (ii)}, the stability of $\Sigma$ allows
    \begin{equation}\label{2.8}
        0 \leq \is|A|^2\varphi^2 \leq \is 2|\nabla \varphi|^2 + (S_{\Sigma} - S_M)\varphi^2
    \end{equation}
    for any smooth function $\varphi$ defined on $\Sigma$.
    Since $S_M \geq k \geq 0$, the inequality (\ref{2.8}) becomes that
    $$k\is \varphi^2 \leq \is 2|\nabla \varphi|^2 + S_{\Sigma} \varphi^2.$$
    Choosing $\varphi=1$ on $\Sigma$ gives
    $$k\,\vol(\Sigma) \leq \is S_{\Sigma}.$$
    Now suppose that the equality holds.
    Since $$\is S_{\Sigma} = k\vol(\Sigma) = \is k \geq 0,$$
    taking $\varphi=1$ on $\Sigma$ and the stability of $\Sigma$ yield
    $$0 \leq \is S_{\Sigma} - S_M -|A|^2 =\is k - S_M - |A|^2.$$
    Then we get
    $$\is S_M + |A|^2 \leq \is k,$$
    and combining with the fact that $k \leq S_M \leq S_M + |A|^2$ implies
    $$S_M=k \quad \text{and} \quad |A|^2=0 \quad \text{on} \quad \Sigma.$$
    Hence $\Sigma$ is totally geodesic.
\end{proof}

\begin{rmk}
{\rm
The converse statements in the equality cases of {\rm (i)} and {\rm (ii)} are false in general. For {\rm (i)}, observe that the round sphere $\mathbb{S}^n$ carries a metric of positive scalar curvature. Hence, by Kazdan-Warner \cite[Theorem 1.3]{KW}, there exists a smooth metric $g_0$ on $\mathbb{S}^n$ such that
$$
S_{g_0}\equiv 0,
$$
where $S_{g_0}$ denotes the scalar curvature of $(\mathbb{S}^n,g_0)$. Now consider the product manifold $M=(\mathbb{S}^n,g_0)\times \mathbb{R}$ and the slice $\Sigma=\mathbb{S}^n\times\{0\}$. Then $\Sigma$ is a compact totally geodesic hypersurface in $M$, and thus it is a stable minimal hypersurface. Moreover, the induced metric on $\Sigma$ is exactly $g_0$. It follows that $S_\Sigma\equiv 0$. However, we have
$$
\sigma(\Sigma)=\sigma(\mathbb{S}^n)>0,
$$
since the smooth Yamabe invariant depends only on the underlying smooth manifold and $\mathbb{S}^n$ admits metrics of positive scalar curvature. Thus the converse in the equality case of {\rm (i)} fails. For {\rm (ii)}, consider the product manifold $M=\mathbb{S}^n\times\mathbb{R}$ with the standard product metric, and let $\Sigma=\mathbb{S}^n\times\{0\}$. Then $\Sigma$ is again compact and totally geodesic. Taking $k=0$, the condition $S_M\ge k$ is trivially satisfied. On the other hand,
$$
\int_\Sigma S_\Sigma\,dv_g
=
n(n-1)\Vol(\Sigma)
>
0
=
k\,\Vol(\Sigma).
$$
Therefore, even if $\Sigma$ is totally geodesic, equality in {\rm (ii)} need not hold. \\

Note that both the cases $\sigma(\Sigma)>0$ and $\sigma(\Sigma)=0$ in {\rm (i)} can indeed occur. For example, $\mathbb{T}^n\times \mathbb{R}$ contains a stable, totally geodesic, flat hypersurface $\mathbb{T}^n$ with $\sigma(\mathbb{T}^n)=0$, where $\mathbb{T}^n$ is the $n$-dimensional torus. Likewise, $\mathbb{S}^n\times \mathbb{R}$ contains a stable hypersurface $\mathbb{S}^n$ with $\sigma(\mathbb{S}^n)>0$, where $\mathbb{S}^n$ is the $n$-dimensional sphere.
}
\end{rmk}

\section{Rigidity and eigenvalue estimates for stable minimal hypersurfaces in Riemannian manifolds with pinched sectional curvature} \label{sec4}

Let $\Sigma$ be a two-sided hypersurface in an $(n+1)$-dimensional Riemannian manifold $M$. Choose a local orthonormal frame $e_1,\ldots,e_{n+1}$ on $M$ such that $e_1,\ldots,e_n$ are tangent to $\Sigma$ and $e_{n+1}=\eta$ is a unit normal vector field along $\Sigma$. We write
$$
h_{ij}=A(e_i,e_j)
$$
for the components of the second fundamental form $A$ of $\Sigma$ and denote by
$$
h_{ijk}:=h_{ij;k}=(\nabla_{e_k}A)(e_i,e_j)
$$
the components of its covariant derivative. With respect to the same frame, we denote by $K_{ijkl}$ the components of the curvature tensor of $M$, and by $K_{ijkl;m}$ their covariant derivatives. We follow the notation in \cite{Schoen&75}. Assume that the sectional curvature of $M$ satisfies
$$
K_1\le K_M\le K_2
$$
for some constants $K_1$ and $K_2$, and that the covariant derivative of the curvature tensor satisfies
$$
|\nabla K|^2
=
\sum_{i,j,k,l,m}K_{ijkl;m}^2
\le K_3^2|A|^2
$$
for some constant $K_3\ge 0$. Define the constant $\Gamma_\Sigma$ by
$$
\Gamma_\Sigma
:=
\sup_\Sigma
\bigl(
-(n^2+5n)K_1+2nK_2+4K_3+|A|^2+S_\Sigma
\bigr),
$$
where $S_\Sigma$ denotes the scalar curvature of $\Sigma$.

The refined Kato inequality used below requires the second fundamental form
to be a Codazzi tensor. Recall that, with respect to the above frame, the
Codazzi equation is
$$
h_{ijk}-h_{ikj}=K_{n+1ikj}.
$$
We therefore assume, when stated in the theorems below, that
$$
K_{n+1ijk}=0
\qquad\text{for all }1\le i,j,k\le n.
$$
Under this assumption, the Codazzi equation implies that the second
fundamental form $A$ is a Codazzi tensor. This condition is automatically
satisfied when $M$ has constant sectional curvature. Since the second fundamental form $A$ is trace-free by minimality, the following lemma is a direct consequence of the refined Kato inequality for trace-free Codazzi tensors due to Bourguignon \cite{Bourguignon81} and will be used later.

\begin{lem}[\cite{Bourguignon81}] \label{lem: Kato}
Let $\Sigma^n$ be a two-sided minimal hypersurface in a Riemannian manifold $M$. Suppose that its second fundamental form $A$ is a Codazzi tensor. Then, at every point where $|A|\neq0$, we have
$$
|\nabla A|^2-|\nabla|A||^2\ge \frac{2}{n}|\nabla|A||^2.
$$
\end{lem}

\begin{proof}
We include the proof for completeness. Fix a point at which $|A|\neq0$ and choose an orthonormal frame diagonalizing $A$, so that $h_{ij}=\lambda_i\delta_{ij}$. Since $A$ is Codazzi, the covariant derivatives $h_{ijk}$ are symmetric in all three indices. Since $\Sigma$ is minimal, $\sum_i h_{iik}=0$ for every $k$. Hence, for each fixed $k$,
$$
h_{kkk}=-\sum_{i\neq k}h_{iik},
$$
and the Cauchy-Schwarz inequality gives
\begin{align}
h_{kkk}^2\le (n-1)\sum_{i\neq k}h_{iik}^2. \label{ineq: Cauchy}
\end{align}
Then, since $h_{ijk}$ is symmetric in all three indices,
$$
|\nabla A|^2
=
\sum_{i,j,k}h_{ijk}^2
\geq
\sum_k
\left(
h_{kkk}^2
+
3\sum_{i\neq k}h_{iik}^2
\right).
$$
For each $k$, it follows from (\ref{ineq: Cauchy}) that
\begin{align*}
&h_{kkk}^2+3\sum_{i\neq k}h_{iik}^2
-
\left(1+\frac{2}{n}\right)
\sum_i h_{iik}^2\\
&=
h_{kkk}^2+3\sum_{i\neq k}h_{iik}^2
-
\left(1+\frac{2}{n}\right)
\left(
h_{kkk}^2+\sum_{i\neq k}h_{iik}^2
\right)\\
&=
\frac{2}{n}
\left(
(n-1)\sum_{i\neq k}h_{iik}^2-h_{kkk}^2
\right)
\geq0.
\end{align*}
Therefore,
$$
h_{kkk}^2+3\sum_{i\neq k}h_{iik}^2
\geq
\left(1+\frac{2}{n}\right)
\sum_i h_{iik}^2.
$$
Summing over $k$ yields
\begin{align}
|\nabla A|^2 \geq \left(1+\frac{2}{n}\right) \sum_{i,k}h_{iik}^2. \label{ineq: A}
\end{align}
On the other hand, since
$$
|A|^2=\sum_i\lambda_i^2,
$$
we have
$$
e_k(|A|)
=
\frac{1}{|A|}
\sum_i\lambda_i h_{iik}.
$$
Thus
\begin{align}
|\nabla|A||^2 &= \frac{1}{|A|^2}\sum_k\left(\sum_i\lambda_i h_{iik}\right)^2 \nonumber \\
&\leq \frac{1}{|A|^2} \sum_k\left(\sum_i\lambda_i^2\right)\left(\sum_i h_{iik}^2\right) \nonumber\\
&=\sum_{i,k}h_{iik}^2. \label{ineq: B}
\end{align}
Combining (\ref{ineq: A}) and (\ref{ineq: B}), we obtain
$$
|\nabla A|^2
\geq
\left(1+\frac{2}{n}\right)|\nabla|A||^2,
$$
which proves the lemma.

\end{proof}

Under a pinching condition on the sectional curvature, we obtain the following rigidity result.

\begin{thm}\label{thmftc1}
Let $M$ be an $(n+1)$-dimensional complete Riemannian manifold whose sectional curvature $K_M$ satisfies $0\le K_1\le K_M\le K_2$ for some constants $K_1$ and $K_2$, and let $\Sigma$ be a complete two-sided stable minimal hypersurface in $M$. Suppose that
$$
\lim_{R\to\infty}\frac{1}{R^2}\int_{B_p(R)}|A|^2=0,
$$
where $B_p(R)$ denotes the geodesic ball in $\Sigma$ of radius $R$ centered at some point $p\in \Sigma$. Assume further that $K_{n+1ijk}=0$ for all $1\le i,j,k\le n$ and that
$$
|\nabla K|^2=\sum_{i,j,k,l,m}K_{ijkl;m}^2\le K_3^2|A|^2
$$
for some constant $K_3\ge 0$. If $\Gamma_\Sigma\le 0$, then $\Sigma$ is totally geodesic.
\end{thm}

\begin{proof}
By (1.22) and (1.27) in \cite{Schoen&75}, we have
\begin{align*}
2|A|\Delta |A|+2|\nabla |A||^2
&\ge 2\sum_{i,j,k} h_{ijk}^2-4K_3|A|^2+2n(2K_1-K_2)|A|^2-2|A|^4 \\
&=2|\nabla A|^2-4K_3|A|^2+2n(2K_1-K_2)|A|^2-2|A|^4.
\end{align*}
Hence
\begin{equation}\label{ki}
|A|\Delta |A|+2K_3|A|^2-n(2K_1-K_2)|A|^2+|A|^4
\ge
|\nabla A|^2-|\nabla |A||^2.
\end{equation}
By the assumption $K_{n+1ijk}=0$ and the Codazzi equation, the second fundamental form $A$ is a Codazzi tensor. Hence,  applying Lemma \ref{lem: Kato} to \eqref{ki}, we obtain
\begin{equation}\label{2}
|A|\Delta |A|+2K_3|A|^2-n(2K_1-K_2)|A|^2+|A|^4\ge \frac{2}{n}|\nabla |A||^2.
\end{equation}
Choose a Lipschitz function $f$ with compact support such that
\begin{equation} \label{fd}
0\leq f\leq 1,\quad f\equiv 1~ \text{on}~ B_p(R),\quad f\equiv 0~\text{on}~\Sigma \backslash B_p(2R), \quad \text{and} \quad|\nabla f|\leq \frac{1}{R},
\end{equation}
where $B_p(R)$ denotes the geodesic ball in $\Sigma$ of radius $R$ centered at $p \in \Sigma$. Multiplying \eqref{2} by $f^2$ and integrating over $\Sigma$, we get
      \begin{align}\label{4}
        \int_\Sigma f^2|A|\lap|A| + 2K_3\int_\Sigma & f^2|A|^2 -n(2K_1-K_2) \int_\Sigma  f^2|A|^2 +\int_\Sigma f^2|A|^4 \nonumber \\
        & \geq \frac{2}{n}\int_\Sigma f^2|\nabla |A||^2.
    \end{align}
Since the divergence theorem gives
      \begin{align*}
        0 & = \int_\Sigma \div(f^2|A|\nabla|A|) \\
         & = \int_\Sigma f^2|A|\lap|A| + \int_\Sigma f^2|\nabla|A||^2 + 2\int_\Sigma f|A|\langle \nabla f,\nabla|A| \rangle,
    \end{align*}
it follows from \eqref{4} that
      \begin{align}\label{5}
        (2K_3-2nK_1+n&K_2)\int_\Sigma f^2|A|^2  +\int_\Sigma f^2|A|^4  \nonumber\\
        & \geq \left(\frac{2}{n}+1\right)\int_\Sigma f^2|\nabla |A||^2 + 2\int_\Sigma f|A|\langle \nabla f,\nabla|A| \rangle.
    \end{align}
On the other hand, the stability of $\Sigma$ implies that
\begin{equation}\label{6}
\int_\Sigma \left( |\nabla \varphi|^2 + \left(\frac{S_\Sigma}{2}-\frac{S_M+|A|^2}{2}\right)\varphi^2 \right)\ge 0
\end{equation}
for every compactly supported Lipschitz function $\varphi$ on $\Sigma$, where $S_M$ denotes the scalar curvature of $M$. Since the sectional curvature of $M$ satisfies
$$
K_1\le K_M\le K_2,
$$
we have
$$
n(n+1)K_1
\le
S_M
=
\Ric_M(e_1,e_1)+\cdots+\Ric_M(e_{n+1},e_{n+1})
\le
n(n+1)K_2.
$$
Thus \eqref{6} becomes
\begin{equation}\label{7}
\int_\Sigma \left(|\nabla \varphi|^2+\frac{S_\Sigma-n(n+1)K_1-|A|^2}{2}\varphi^2\right)\ge 0.
\end{equation}
Replacing $\varphi$ by $|A|f$ in (\ref{7}), we have
\begin{align}\label{8}
\int_\Sigma |A|^2|\nabla f|^2 + \int_\Sigma & f^2|\nabla |A||^2 +  2\int_\Sigma f|A|\langle \nabla f,\nabla|A| \rangle \nonumber \\
&=\int_\Sigma |\nabla (f|A|)|^2 \nonumber \\
& \geq \int_\Sigma \frac{1}{2}f^2|A|^4 + \int_\Sigma \frac{n(n+1)K_1-S_{\Sigma}}{2}f^2|A|^2.
\end{align}
Combining the inequalities (\ref{5}) and (\ref{8}) gives
\begin{align}\label{9}
\int_\Sigma &\left(2K_3-2nK_1+nK_2 +\frac{|A|^2-n(n+1)K_1+S_{\Sigma}}{2}\right) f^2|A|^2 +  \int_\Sigma |A|^2|\nabla f|^2 \nonumber\\
& \geq \frac{2}{n}\int_\Sigma f^2|\nabla |A||^2.
\end{align}
Then the assumption $S_{\Sigma} \leq (n^2+5n)K_1-2n K_2-4 K_3-|A|^2$ allows
\begin{align}
\int_\Sigma |A|^2|\nabla f|^2 \geq \frac{2}{n}\int_\Sigma f^2|\nabla |A||^2. \label{ineq: stability}
\end{align}
Since $|\nabla f|\le 1/R$ and $\operatorname{supp}f\subset B_p(2R)$, we have
$$
\int_\Sigma |A|^2|\nabla f|^2
\le
\frac{1}{R^2}\int_{B_p(2R)}|A|^2
=
4\frac{1}{(2R)^2}\int_{B_p(2R)}|A|^2\rightarrow0.
$$
as $R \rightarrow \infty$ by assumption. Because $f\equiv1$ on $B_p(R)$, (\ref{ineq: stability}) implies
$$
\int_{B_p(R)}|\nabla|A||^2\rightarrow0,
$$
and hence $\nabla|A|\equiv0$ on  $\Sigma$. Thus $|A|\equiv c$ for some constant $c\ge0$. We claim that $c=0$. Suppose, to the contrary, that $c>0$. Applying the stability inequality to the same cut-off function $f$ and using $\Ric_M(\eta,\eta)\ge nK_1\ge0$, we obtain
$$
c^2\int_\Sigma f^2
\le
\int_\Sigma
\bigl(|A|^2+\Ric_M(\eta,\eta)\bigr)f^2
\le
\int_\Sigma|\nabla f|^2.
$$
Since $f\equiv1$ on $B_p(R)$ and $\operatorname{supp}f\subset B_p(2R)$, we see that
\begin{align}
c^2\Vol(B_p(R)) \le \frac{1}{R^2}\Vol(B_p(2R)) = \frac{1}{c^2R^2}\int_{B_p(2R)}|A|^2 \rightarrow0. \label{ineq: aa}
\end{align}
On the other hand, for any fixed $R_0>0$ and every $R\ge R_0$, 
$$c^2\Vol(B_p(R))\ge c^2\Vol(B_p(R_0))>0,$$
which is a contradiction by (\ref{ineq: aa}). Therefore $c=0$, which shows that $\Sigma$ is totally geodesic.

\end{proof}

\begin{rmk}
\rm
Under the sectional curvature assumption in Theorem~\ref{thmftc1}, we have
$$
nK_1\le \Ric_M(\eta,\eta)\le nK_2
$$
and
$$
n(n+1)K_1\le S_M\le n(n+1)K_2.
$$
Hence, by the Gauss equation,
\begin{align*}
-2nK_2
&\le -2\Ric_M(\eta,\eta) \\
&= -S_M+S_\Sigma+|A|^2 \\
&\le -n(n+1)K_1+S_\Sigma+|A|^2 \\
&\le 4nK_1-2nK_2-4K_3,
\end{align*}
where in the last step we used the assumption $\Gamma_\Sigma\le 0$. It follows that
$$
K_1\ge \frac{K_3}{n}\ge 0.
$$
In particular, in Theorem~\ref{thmftc1}, the constants $K_1$ and $K_2$ must be nonnegative.
\end{rmk}



In the Euclidean case $M=\mathbb{R}^{n+1}$, the sectional curvature vanishes identically and the ambient space is locally symmetric. Therefore, the curvature constants in Theorem~\ref{thmftc1} satisfy $K_1=K_2=K_3=0$. Moreover, the Gauss equation yields $S_\Sigma=-|A|^2,$
which shows that
$$\Gamma_\Sigma = 0.$$ Thus the following corollary follows immediately from Theorem~\ref{thmftc1}.

\begin{cor}[\cite{CP80}]\label{CP}
Let $\Sigma$ be a complete oriented stable minimal hypersurface in $\mathbb{R}^{n+1}$ satisfying
$$
\int_\Sigma |A|^2<\infty.
$$
Then $\Sigma$ is a hyperplane.
\end{cor}

\medskip

Let $D$ be a compact domain in a complete noncompact manifold $\Sigma$. We recall that the first eigenvalue $\lambda_1(D)$ of the  Dirichlet boundary value problem
$$
\begin{cases}
\Delta \phi+\lambda\phi=0 & \text{in } D,\\
\phi=0 & \text{on } \partial D
\end{cases}
$$
is characterized variationally by
$$
\lambda_1(D)=\inf_\phi \frac{\int_D |\nabla \phi|^2}{\int_D \phi^2},
$$
where the infimum is taken over all smooth functions $\phi$ on $D$ vanishing on $\partial D$. The first eigenvalue $\lambda_1(\Sigma)$ of $\Sigma$ is then defined by
\begin{equation*}
\lambda_1(\Sigma)=\inf_D \lambda_1(D),
\end{equation*}
where $D$ ranges over all relatively compact domains in $\Sigma$. Under a pinching assumption on the sectional curvature, we obtain the following eigenvalue estimate.

\begin{thm}\label{thmftc2}
Let $M$ be an $(n+1)$-dimensional complete Riemannian manifold whose sectional curvature $K_M$ satisfies $K_1\leq K_M\leq K_2$ for some constants $K_1$ and $K_2$, and let $\Sigma$ be a complete two-sided stable non-totally geodesic minimal hypersurface in $M$. Suppose that
$$
\lim_{R\to\infty}\frac{1}{R^2}\int_{B_p(R)}|A|^2=0,
$$
where $B_p(R)$ denotes the geodesic ball in $\Sigma$ of radius $R$ centered at some point $p\in\Sigma$. Assume further that $K_{n+1ijk}=0$ for all $1\leq i,j,k\leq n$ and 
$$
|\nabla K|^2=\sum_{i,j,k,l,m}K_{ijkl;m}^2\leq K_3^2|A|^2
$$
for some constant $K_3\geq 0$. If $\Gamma_{\Sigma}>0$, then we have
$$
\lambda_1(\Sigma)\leq \frac{n}{4}\Gamma_{\Sigma}.
$$
\end{thm}

\begin{proof}
If $\Gamma_{\Sigma}=\infty$, then the conclusion is trivial. Thus we may assume that $\Gamma_{\Sigma}<\infty$. By the variational definition of $\lambda_1(\Sigma)$, for every nonzero compactly supported Lipschitz function $\phi$ on $\Sigma$,
    \begin{equation}\label{10}
        \lambda_1(\Sigma) \leq \frac{\int_\Sigma|\nabla \phi|^2}{\int_\Sigma \phi^2}.
    \end{equation}
Choose $\phi=|A|f$ in \eqref{10}, where $f$ is the cut-off function defined in \eqref{fd}. Then
  \begin{align}\label{11}
\lambda_1(\Sigma)\int_\Sigma f^2|A|^2
&\le \int_\Sigma |\nabla (|A|f)|^2 \nonumber\\
&= \int_\Sigma f^2|\nabla |A||^2
 + \int_\Sigma |A|^2|\nabla f|^2
 + 2\int_\Sigma f|A|\langle \nabla f,\nabla |A| \rangle.
\end{align}
Moreover, Young's inequality gives
  \begin{equation}\label{12}
        2\int_\Sigma f|A|\langle \nabla f, \nabla |A|\rangle
        \leq \epsilon\int_\Sigma f^2|\nabla |A||^2 + \frac{1}{\epsilon} \int_\Sigma |A|^2|\nabla f|^2
    \end{equation}
    for every $\epsilon>0$. Plugging (\ref{12}) into (\ref{11}), we obtain
  \begin{equation}\label{13}
\lambda_1(\Sigma)\int_\Sigma f^2|A|^2
\le
\left(1+\epsilon\right)\int_\Sigma f^2|\nabla |A||^2
+ \left(1+\frac{1}{\epsilon}\right)\int_\Sigma |A|^2|\nabla f|^2.
\end{equation}
On the other hand, using the definition of $\Gamma_\Sigma$ and  (\ref{9}) gives
      \begin{align}\label{14}
        \frac{\Gamma_{\Sigma}}{2}&\int_\Sigma f^2|A|^2 + \int_\Sigma |A|^2|\nabla f|^2 \nonumber\\
        &\geq \frac{1}{2}\int_\Sigma(-(n^2+5n)K_1 + 2nK_2 + 4K_3 + |A|^2 + S_{\Sigma}) f^2|A|^2 + \int_\Sigma |A|^2|\nabla f|^2 \nonumber \\
        & = \int_\Sigma\left(2K_3-2nK_1+nK_2+\frac{|A|^2-n(n+1)K_1+S_{\Sigma}}{2}\right) f^2|A|^2 + \int_\Sigma |A|^2|\nabla f|^2 \nonumber\\
        & \geq \frac{2}{n}\int_\Sigma f^2|\nabla |A||^2.
\end{align}
If $\lambda_1(\Sigma)=0$, then the desired estimate holds trivially since $\Gamma_\Sigma>0$. Hence we may assume that $\lambda_1(\Sigma)>0$. Combining \eqref{13} and \eqref{14} and using the assumption that $\Gamma_\Sigma>0$, we see that
  \begin{equation*}
\left\{
1+\frac{\Gamma_\Sigma}{2\lambda_1(\Sigma)}\left(1+\frac{1}{\epsilon}\right)
\right\}
\int_\Sigma |A|^2|\nabla f|^2
\ge
\left\{
\frac{2}{n}-\frac{\Gamma_\Sigma}{2\lambda_1(\Sigma)}(1+\epsilon)
\right\}
\int_\Sigma f^2|\nabla |A||^2.
\end{equation*}
Now suppose, for contradiction, that
$$
\lambda_1(\Sigma)>\frac{n}{4}\Gamma_\Sigma.
$$
Then we may choose $\epsilon>0$ sufficiently small so that
$$
\frac{2}{n}-\frac{\Gamma_\Sigma}{2\lambda_1(\Sigma)}(1+\epsilon)>0.
$$
Since $|\nabla f|\le 1/R$ and $\operatorname{supp}f\subset B_p(2R)$, the growth assumption gives
$$
\int_\Sigma |A|^2|\nabla f|^2
\le
\frac{1}{R^2}\int_{B_p(2R)}|A|^2\longrightarrow0.
$$
  Because $f\equiv1$ on $B_p(R)$ and the coefficient on the right-hand side is strictly positive, we obtain $\nabla|A|\equiv0$. Hence $|A|\equiv c$ for some constant $c\ge0$.

If $c>0$, applying \eqref{10} to the cut-off function $f$ itself gives
$$
\lambda_1(\Sigma)\int_\Sigma f^2
\le
\int_\Sigma|\nabla f|^2
\le
\frac{1}{R^2}\Vol(B_p(2R))
=
\frac{1}{c^2R^2}\int_{B_p(2R)}|A|^2\longrightarrow0.
$$
Since $f\equiv1$ on $B_p(R)$, the left-hand side is bounded below by $\lambda_1(\Sigma)\Vol(B_p(R_0))>0$ for any fixed $R_0$ and all sufficiently large $R$, a contradiction. Thus $c=0$, contrary to the assumption that $\Sigma$ is not totally geodesic. This completes the proof.
\end{proof}

In the hyperbolic setting, the ambient manifold $\mathbb{H}^{n+1}$ has constant sectional curvature $-1$ and is locally symmetric. Hence $K_1=K_2=-1$ and $K_3=0$. Moreover, the Gauss equation yields $S_\Sigma=-n(n-1)-|A|^2$, which gives
$$\Gamma_\Sigma = 4n.$$
Therefore, together with the lower bound of Cheung and Leung \cite{C&L01}, Theorem~\ref{thmftc2} immediately gives the following corollary.
\begin{cor}[\cite{Seo11}]\label{S}
Let $\Sigma$ be a complete oriented stable minimal hypersurface in $\mathbb{H}^{n+1}$ satisfying
$$
\int_\Sigma |A|^2<\infty.
$$
Then we have
$$
\frac{(n-1)^2}{4}\le \lambda_1(\Sigma)\le n^2.
$$
\end{cor}

\section{Stability and angle function of minimal hypersurfaces in warped product manifolds} \label{sec5}

Let $(N,g_N)$ be an $n$-dimensional Einstein manifold for $n\ge 2$. Then there exists a constant $k$ such that
\begin{equation}\label{rcm}
\Ric_N=(n-1)k\,g_N.
\end{equation}
For such $(N,g_N)$, we consider the $(n+1)$-dimensional   warped product
$$
M=(0,\bar r)\times_h N,
$$
where $0<\bar r\le\infty$, endowed with the metric
$$
g=dr^2+h(r)^2g_N,
$$
where $h:(0,\bar r)\to\mathbb{R}$ is a smooth positive warping function. We do not assume that the ambient warped product is complete, since the argument is local in the ambient geometry. For the space form applications, after fixing a pole $p$, the punctured Euclidean and hyperbolic spaces have the standard representations
$$
\mathbb{R}^{n+1}\setminus\{p\}=(0,\infty)\times_r\mathbb{S}^n,
\qquad
\mathbb{H}^{n+1}\setminus\{p\}=(0,\infty)\times_{\sinh r}\mathbb{S}^n.
$$
The open hemisphere can similarly be written as $(0,\pi/2)\times_{\sin r}\mathbb{S}^n$ after removing its center point.
Let $\{\epsilon_1,\ldots,\epsilon_n\}$ be a local orthonormal frame on $N$ with respect to the metric $g_N$. Then
$$
\left\{\frac{\epsilon_1}{h},\ldots,\frac{\epsilon_n}{h},\frac{\partial}{\partial r}\right\}
$$
forms a local orthonormal frame on $M$. From now on, we shall work with the orthonormal frame
\begin{equation*}\label{b1}
\beta=\left\{\frac{\epsilon_1}{h},\ldots,\frac{\epsilon_n}{h},\frac{\partial}{\partial r}\right\}.
\end{equation*}
Denote by $\overline{\nabla}$ the Levi-Civita connection of $(M,g)$. For the conformal vector field $X=h(r)\frac{\partial}{\partial r},$
Brendle \cite[Lemma~2.2]{brendle13} proved that
\begin{equation*}
\overline{\nabla}_T X=h'(r)\,T
\end{equation*}
for every vector field $T\in \mathfrak{X}(M)$.

Let $\Sigma$ be an isometrically immersed hypersurface in $M$. Throughout this section, we assume that $\Sigma$ is \textit{two-sided}, that is, it admits a globally defined unit normal vector field $\eta$. Define the \textit{angle function} $\nu$ on $\Sigma$ by
$$
\nu=g\!\left(\frac{\partial}{\partial r},\,\eta\right).
$$
The radial vector field $\partial/\partial r$ is globally defined on the warped-product region $(0,\bar r)\times_h N$. In the space form case below, we fix a pole $p\notin\Sigma$ and let $r$ denote the distance from $p$. Throughout the remainder of this section, we  assume that $\nu$ does not change sign on $\Sigma$, and we choose the unit normal $\eta$ so that $\nu\ge 0$ on $\Sigma$. For instance, if $\Sigma$ is locally given as a graph, then $\nu>0$ on $\Sigma$. According to the proof of Proposition~2.1 in \cite{brendle13}, the Ricci curvature $\Ric_M$ of $M$ is given by
$$
\Ric_M
=
\Ric_N
-\bigl(h(r)h''(r)+(n-1)h'(r)^2\bigr)g_N
-n\frac{h''(r)}{h(r)}\,dr\otimes dr.
$$
Hence
\begin{align*}
    \Ricm & =\Ricn - \frac{h(r)h''(r)+(n-1)h'(r)^2}{h(r)^2}(g-dr\otimes dr) - n\frac{h''(r)}{h(r)}dr\otimes dr \nonumber\\
    &=\Ricn -\left(n\left(\frac{h'(r)}{h(r)}\right)^2+(\log h)''(r)\right)g - (n-1)(\log h)''(r)\,dr\otimes dr,
\end{align*}
which is equivalent to
\begin{align*}
    \Ricm(U,V)  &  = \Ricn(U^N,V^N) - \left(n\left(\frac{h'(r)}{h(r)}\right)^2+(\log h)''(r)\right)g(U,V) \\
    & \quad - (n-1)(\log h)''(r)\,dr\otimes dr(U\otimes V),
\end{align*}
where
$$
U^N = U - g\left(U,\ddr\right)\ddr, \quad
V^N = V - g\left(V,\ddr\right)\ddr
$$
denote the projections of the vector fields $U,V \in \mathfrak{X}(M)$ onto $N$, respectively. In particular,
\begin{align}
    \Ricm(\eta,\eta) &= \Ricn(\eta^N,\eta^N) - \left(n\left(\frac{h'(r)}{h(r)}\right)^2+(\log h)''(r)\right) \nonumber \\
    & \quad - (n-1)(\log h)''(r)\,\nu^2, \label{Riceta}
\end{align}
where $$\eta^N = \eta-g\left(\eta,\ddr\right)\ddr$$
is the projection of $\eta$ on $N$. Recall that $\{\epsilon_1,\ldots,\epsilon_n\}$ is a local orthonormal frame on $N$ with respect to the metric $g_N$, and that
$$
\beta=\left\{\frac{\epsilon_1}{h(r)},\ldots,\frac{\epsilon_n}{h(r)},\frac{\partial}{\partial r}\right\}
$$
is a local orthonormal frame on $M$ with respect to the metric $g$. Then the vectors $\eta^N$ and $\eta$ may be written in terms of these frames as
$$
\eta^N=\sum_{i=1}^n g_N(\eta^N,\epsilon_i)\epsilon_i
      =\sum_{i=1}^n \frac{1}{h(r)}\,g\!\left(\eta^N,\frac{\epsilon_i}{h(r)}\right)\epsilon_i,
$$
and
$$
\eta=\sum_{i=1}^n g\!\left(\eta,\frac{\epsilon_i}{h(r)}\right)\frac{\epsilon_i}{h(r)}
     +g\!\left(\eta,\frac{\partial}{\partial r}\right)\frac{\partial}{\partial r}.
$$
Thus, using \eqref{rcm}, we obtain
\begin{align*}
\Ric_N(\eta^N,\eta^N)
&= (n-1)k\, g_N(\eta^N,\eta^N) \\
&= (n-1)k \sum_{i,j} g_N(\eta^N,\epsilon_i)\, g_N(\eta^N,\epsilon_j)\, g_N(\epsilon_i,\epsilon_j) \\
&= (n-1)\frac{k}{h(r)^2}\sum_{i,j}
g\!\left(\eta^N,\frac{\epsilon_i}{h(r)}\right)
g\!\left(\eta^N,\frac{\epsilon_j}{h(r)}\right)\delta_{ij} \\
&= (n-1)\frac{k}{h(r)^2}\sum_i
g\!\left(\eta,\frac{\epsilon_i}{h(r)}\right)^2 \\
&= (n-1)\frac{k}{h(r)^2}(1-\nu^2).
\end{align*}
Combining this with \eqref{Riceta}, we have
\begin{align}\label{ric}
    \Ricm(\eta,\eta)
    & = (n-1)\frac{k}{h(r)^2}(1-\nu^2) \nonumber \\
    & \quad
    - \left(n\left(\frac{h'(r)}{h(r)}\right)^2+(\log h)''(r)\right) - (n-1)(\log h)''(r)\nu^2 \nonumber \\
    & = (n-1)\frac{k}{h(r)^2}(1-\nu^2) \nonumber \\
    & \quad -\left(n\left(\frac{h'(r)}{h(r)}\right)^2+(\log h)''(r)+(n-1)(\log h)''(r)\right) + (n-1)(\log h)''(r) \nonumber\\
    & \quad -(n-1)(\log h)''(r)\nu^2 \nonumber\\
    & = (n-1)\frac{k}{h(r)^2}(1-\nu^2) - n\frac{h''(r)}{h(r)} + (n-1)(\log h)''(r)(1-\nu^2)\nonumber \\
    & = (n-1)\left(\frac{k}{h(r)^2} + (\log h)''(r)\right)(1-\nu^2) - n\frac{h''(r)}{h(r)}.
\end{align}
On the other hand, let $\{e_1,\ldots,e_n\}$ be a local orthonormal frame on $\Sigma$ with respect to the metric induced from $M$. Since $\eta$ is the unit normal vector field along $\Sigma$, we set $e_{n+1}=\eta$. Then ${e_1,\ldots,e_n,e_{n+1}}$ forms a local orthonormal frame on $M$. Following the notation of Section~\ref{sec4}, we denote by $K_{ijkl}$ the components of the curvature tensor of $M$. In particular,
$$
K_{n+1ijk}=g\bigl(\eta,R(e_j,e_k)e_i\bigr),
$$
where $R$ denotes the Riemann curvature tensor of $M$, given by
$$
R(e_j,e_k)e_i
=
\overline{\nabla}_{e_j}\overline{\nabla}_{e_k}e_i
-
\overline{\nabla}_{e_k}\overline{\nabla}_{e_j}e_i
-
\overline{\nabla}_{[e_j,e_k]}e_i.
$$
Here $\overline{\nabla}$ is the Levi-Civita connection of $(M,g)$. Then
\begin{align}
    -\sum_{i,j}K_{n+1iij}\,g(X,e_j)
    & = \sum_{i,j} g\left(\eta,R(e_j,e_i)e_i\right) \,g(X,e_j) \nonumber \\
    & = \sum_j \Ricm(\eta,e_j) \, g(X,e_j) \nonumber\\
    & = \sum_j \bigg(\Ricn \left(\eta^N,e_j^N\right)  - (n-1)(\log h)''(r) \, g\left(\eta,\ddr\right) \, g\left(e_j,\ddr\right)\bigg) \, g(X,e_j) \nonumber\\
    & = \sum_j \bigg(\Ricn\left(\eta^N,e_j^N\right)g(X,e_j) \nonumber\\
    &  \quad - (n-1)(\log h)''(r) \, g\left(\eta,h(r)\ddr\right) \, g\left(e_j,\ddr\right) \, g\left(\frac{X}{h(r)},e_j\right) \nonumber\bigg)\\
    & = \sum_j \Ricn\left(\eta^N,e_j^N\right) \, g(X,e_j) - (n-1)(\log h)''(r) \, g(X,\eta) \sum_j g\left(e_j,\ddr\right)^2 \nonumber \\
    & = \sum_j \Ricn\left(\eta^N,e_j^N\right) \, g(X,e_j) - (n-1)(1-\nu^2)(\log h)''(r) \, g(X,\eta) \nonumber
\end{align}
On the other hand, we note that
$$\eta^N = \sum_{i=1}^{n} g_N\left(\eta^N,\epsilon_i\right)\epsilon_i, \quad e_j^N = \sum_{k=1}^{n} g_N\left(e_j^N,\epsilon_k\right)\epsilon_k.$$
Moreover,
\begin{align*}
    \sum_j g\left(e_j^N,\frac{\epsilon_i}{h(r)}\right) \, g(X,e_j) & = g\left(X,\sum_jg\left(e_j^N,\frac{\epsilon_i}{h(r)}\right)e_j\right) \nonumber \\
    & = g\left(X,\sum_j \, g\left(e_j,\frac{\epsilon_i}{h(r)}\right)e_j\right) \nonumber \\
    & = g\bigg(X,\frac{\epsilon_i}{h(r)}-g\left(\eta,\frac{\epsilon_i}{h(r)}\right)\eta\bigg) \nonumber \\
    & = g\bigg(h(r)\ddr,\frac{\epsilon_i}{h(r)}\bigg) - g\bigg(X, g\left(\eta,\frac{\epsilon_i}{h(r)}\right)\eta\bigg) \nonumber \\
    & = -g\left(\eta,\frac{\epsilon_i}{h(r)}\right)\,g(X,\eta).
\end{align*}
Using these relations, we have
\begin{align}\label{5.5}
    \sum_j \Ricn  \left(\eta^N,e_j^N\right)\,g(X,e_j)
    & = \sum_j(n-1) k g_N\left(\eta^N,e_j^N\right)\,g(X,e_j) \nonumber \\
    & = \sum_{i,j,k} (n-1)k g_N\left(\eta^N,\epsilon_i\right) g_N\left(e_j^N,\epsilon_k\right) g_N(\epsilon_i,\epsilon_k) \, g(X,e_j) \nonumber \\
    & = (n-1)k \sum_{i,j} g_N\left(\eta^N,\epsilon_i\right) g_N\left(e_j^N,\epsilon_i\right)\,g(X,e_j) \nonumber \\
    & = (n-1)\frac{k}{h(r)^2} \sum_{i} g\left(\eta^N,\frac{\epsilon_i}{h(r)}\right) \sum_{j} g\left(e_j^N,\frac{\epsilon_i}{h(r)}\right)\, g(X,e_j)  \nonumber \\
    & = -(n-1)\frac{k}{h(r)^2}\sum_{i} g\left(\eta,\frac{\epsilon_i}{h(r)}\right) \, g\left(\eta,\frac{\epsilon_i}{h(r)}\right)\,g(X,\eta) \nonumber \\
    & = -(n-1)\frac{k}{h(r)^2}(1-\nu^2)\,g(X,\eta).
\end{align}
Therefore
\begin{align}\label{Kiij}
    -\sum_{i,j}K_{n+1iij}\,g(X,e_j)
    & = \sum_j \Ricn\left(\eta^N,e_j^N\right) \, g(X,e_j) - (n-1)(1-\nu^2)(\log h)''(r) \, g(X,\eta) \nonumber\\
    & = -(n-1)(1-\nu^2)\left(\frac{k}{h(r)^2} \, g(X,\eta) + (\log h)''(r) \, g(X,\eta)\right).
\end{align}
Denote by $\nabla$ and $\Delta$ the gradient and the Laplacian operator on $\Sigma$, respectively. It is well known that
$$\lap \eta = -n\nabla H - |A|^2\eta - \sum_{i,j}K_{n+1iij}e_j.$$
Then it follows from (\ref{Kiij}) that
\begin{align}\label{lapg}
    \lap g(X,\eta) & = \sum_i(e_i\,e_i-\nabla_{e_i}e_i)\,g(X,\eta) \nonumber \\
    & = \sum_ie_i\bigg(g\left(\Grad_{e_i}X,\eta\right) + g\left(X,\Grad_{e_i}\eta\right)\bigg) - \sum_i\bigg( g\left(\Grad_{\nabla_{e_i}e_i}X,\eta\right) + g\left(X,\Grad_{\nabla_{e_i}e_i}\eta\right) \bigg) \nonumber \\
    & = \sum_ie_i\bigg(g\left(h'(r)e_i,\eta\right) + g\left(X,\Grad_{e_i}\eta\right)\bigg) - \sum_i\bigg(g\left(h'(r)\nabla_{e_i}e_i,\eta\right) + g\left(X,\Grad_{\nabla_{e_i}e_i}\eta\right)\bigg) \nonumber \\
    & = \sum_i g\left(\Grad_{e_i}X,\Grad_{e_i}\eta\right) + \sum_i g\left(X,\Grad_{e_i}\Grad_{e_i}\eta\right) - \sum_i g\left(X,\Grad_{\nabla_{e_i}e_i}\eta\right) \nonumber \\
    & = -n h'(r) H + g\bigg(X,\left(\Grad_{e_i}\Grad_{e_i}-\nabla_{e_i}e_i\right)\eta\bigg) \nonumber \\
    & = -n h'(r) H + g(X,\lap \eta) \nonumber \\
    & = -n h'(r) H + g\left(X,-n\nabla H - |A|^2\eta - \sum_{i,j}K_{n+1iij}e_j\right) \nonumber \\
    & = -n h'(r) H -n g\left(X,\nabla H\right) -g(X,\eta)|A|^2 -\sum_{i,j}K_{n+1iij}\,g(X,e_j) \nonumber \\
    & = -n h'(r) H -n g\left(X,\nabla H\right) -g(X,\eta)|A|^2  \nonumber \\
    & \quad -(n-1)(1-\nu^2)\left(\frac{k}{h(r)^2} \, g(X,\eta) + (\log h)''(r) \, g(X,\eta)\right),
\end{align}
where we used the fact that $e_i$ and $\nabla_{e_i}e_i$ are perpendicular to $\eta$ in the fourth equality. As a direct consequence of the above computations, we have the following result.

\begin{lem}\label{lem4.2}
Let $M=(0,\bar{r})\times_h N$ be an $(n+1)$-dimensional  warped product manifold with the metric $g=dr^2+h(r)^2g_N$, where $n\ge 2$ and $h(r)>0$ on $(0,\bar{r})$. Assume that $N$ is an $n$-dimensional Einstein manifold with $\mathrm{Ric}_N = (n-1)k g_N$ for some constant $k$. Let $X=h(r)\frac{\partial}{\partial r}$ be the conformal radial vector field. If $\Sigma$ is a two-sided minimal hypersurface immersed in $M$, then
 $$L\,g(X,\eta) = -n\frac{h''(r)}{h(r)}g(X,\eta),$$
where $L = \lap + |A|^2 + \mathrm{Ric}_M(\eta,\eta)$ denotes the Jacobi operator.
\end{lem}

\begin{proof}
Since $\Sigma$ is minimal, it follows from \eqref{ric} and \eqref{lapg} that
    \begin{align*}
        L\,g(X,\eta) & = \left(\lap + |A|^2 + \Ricm(\eta,\eta)\right)\,g(X,\eta)\nonumber \\
        & = -nh'H - n \, g(X,\nabla H) - g(X,\eta)|A|^2 \nonumber \\
        & \quad - (n-1)(1-\nu^2)\left(\frac{k}{h(r)^2} \, g(X,\eta) + (\log h)''(r) \, g(X,\eta)\right) \nonumber \\
        & \quad + g(X,\eta)|A|^2 \nonumber \\
        & \quad + (n-1)\left(\frac{k}{h(r)^2} + (\log h)''(r)\right)(1-\nu^2)g(X,\eta) - n\frac{h''(r)}{h(r)}g(X,\eta) \nonumber\\
        & =-n\frac{h''(r)}{h(r)}g(X,\eta).
    \end{align*}
\end{proof}
We are now ready to prove the following theorem.

\begin{thm} \label{thmw}
Let $M=(0,\bar{r})\times_h N$ be an $(n+1)$-dimensional  warped product manifold with the metric $g=dr^2+h(r)^2g_N$, where $n\ge 2$ and $h(r)>0$ on $(0,\bar{r})$. Assume that $N$ is an $n$-dimensional Einstein manifold with $\mathrm{Ric}_N = (n-1)k g_N$ for some constant $k$. Let $\Sigma$ be a complete noncompact two-sided minimal hypersurface immersed in $M$ with $\nu\geq0$, where $\nu=g\!\left(\frac{\partial}{\partial r},\eta\right)$ is the angle function. If $h''(r)\geq0$ on $\Sigma$, then either $\nu>0$ everywhere on $\Sigma$ or $\nu\equiv 0$ on $\Sigma$. Moreover, if $\nu>0$ on $\Sigma$, then $\Sigma$ is stable.
\end{thm}

\begin{proof}
 Let $X=h(r)\frac{\partial}{\partial r}$ and define the nonnegative function $f$ on $\Sigma$  by
$$
f=g(X,\eta)=h(r)\nu.
$$
By Lemma~\ref{lem4.2}, we have
$$
Lf=L(g(X,\eta))=-n\frac{h''(r)}{h(r)}g(X,\eta)=-nh''(r)\nu\le 0.
$$
Suppose that there exists a point $p_0\in\Sigma$ such that $f(p_0)=0$. Let $V(p_0)$ be a relatively compact open neighborhood of $p_0$ in $\Sigma$. Define the elliptic operator $L'=\lap+m$, where
$$
m=\min\left\{0,\inf_{V(p_0)}\bigl(|A|^2+\Ricm(\eta,\eta)\bigr)\right\}.
$$
Then we obtain
\begin{align*}
L'(-f)
&= -\lap f-mf \\
&\ge -\lap f-\bigl(|A|^2+\Ricm(\eta,\eta)\bigr)f \\
&= -Lf \\
&\ge 0.
\end{align*}
on $V(p_0)$. 
By the maximum principle (see \cite{G&T01} for example), it follows that $f$ vanishes on $V(p_0)$. Hence, by the unique continuation property, we conclude that $f \equiv 0$ on $\Sigma$, that is, $\nu \equiv 0$ on $\Sigma$. If $\nu>0$, then we have
$$
g(X,\eta)=h(r)\,g\left(\frac{\partial}{\partial r},\eta\right)=h(r)\,\nu>0.
$$
Therefore, the conclusion follows from Corollary~\ref{fc&s} and Lemma~\ref{lem4.2}.
\end{proof}

As an immediate consequence of Theorem~\ref{thmw}, we obtain the following corollary for Euclidean space and hyperbolic space.

\begin{cor}\label{wceh}
Let $\Sigma$ be a complete noncompact two-sided minimal hypersurface immersed in $\mathbb{R}^{n+1}$ or $\mathbb{H}^{n+1}$ for $n\ge2$. Fix a point $p\notin\Sigma$, let $r$ denote the distance from $p$, and let $\partial/\partial r$ be the corresponding radial unit vector field on the punctured ambient space. Assume that
$$
\nu_p=g\!\left(\frac{\partial}{\partial r},\eta\right)\ge0
$$
on $\Sigma$. 
Then either $\nu_p>0$ everywhere on $\Sigma$ or $\nu_p\equiv0$ on $\Sigma$. 
Moreover, if $\nu_p>0$ on $\Sigma$, then $\Sigma$ is stable.
\end{cor}

\begin{proof}
Because $p\notin\Sigma$, the hypersurface is contained in the punctured ambient space. With respect to the fixed pole $p$, we have
$$
\mathbb{R}^{n+1}\setminus\{p\}=(0,\infty)\times_r\mathbb{S}^n,
\qquad
\mathbb{H}^{n+1}\setminus\{p\}=(0,\infty)\times_{\sinh r}\mathbb{S}^n.
$$
Thus the warping function is respectively $h(r)=r$ or $h(r)=\sinh r$, and in either case $h''(r)\ge0$. The desired conclusion  follows from Theorem~\ref{thmw}.
\end{proof}

\vskip 0.3cm
\noindent
{\bf Acknowledgment: } The authors would like to thank the referee for the careful reading of the manuscript and for the valuable comments and suggestions that helped improve the paper. The second author was supported by the National Research Foundation of Korea (NRF-2021R1A2C1003365).


\vskip 1cm
\noindent Jooyeon Park\\
Department of Mathematics\\
Sookmyung Women's University \\
Cheongpa-ro 47-gil 100, Yongsan-ku, Seoul, 04310, Korea \\
{\tt e-mail: yeonpark@sookmyung.ac.kr}\\

\bigskip
\noindent Keomkyo Seo\\
Department of Mathematics and Research Institute of Natural Sciences\\
Sookmyung Women's University\\
Cheongpa-ro 47-gil 100, Yongsan-ku, Seoul, 04310, Korea \\
{\tt E-mail: kseo@sookmyung.ac.kr}\\
URL: http://sites.google.com/site/keomkyo/


\begin{thebibliography}{123}

\bibitem{A&R16} J. A. Aledo, R. M. Rubio, \textit{Stable minimal surfaces in Riemannian warped products}, J. Geom. Anal. {\bf 27} (1) (2017), 65-78.

\bibitem{A66} F. J. Almgren, Jr., \textit{Some interior regularity theorems for minimal surfaces and an extension of Bernstein’s theorem}, Ann. of Math. {\bf 84} (2) (1966) 277–292.

\bibitem{B27} S. Bernstein, \textit{Über ein geometrisches Theorem und seine Anwendung auf die partiellen Differentialgleichungen vom elliptischen Typus}, Math. Z. {\bf 26} (1927), 551–558.

\bibitem{brendle13} S. Brendle, \textit{Constant mean curvature surfaces in warped product manifolds}, Publ. math. IHES {\bf 117} (2013), 247-269.

\bibitem{BGG69} E. Bombieri, E. De Giorgi, E. Giusti, \textit{Minimal cones and the Bernstein problem}, Invent. Math. {\bf 7} (1969), 243-268.


\bibitem{Bourguignon81} J.-P. Bourguignon, \textit{Codazzi tensor fields and curvature operators}, Global Differential Geometry and Global Analysis, Lecture Notes in Math. {\bf 838}, Springer, Berlin, 1981, 249-250.

\bibitem{CMR24} G. Catino, P. Mastrolia, A. Roncoroni, \textit{Two rigidity results for stable minimal hypersurfaces}, Geom. Funct. Anal. {\bf 34} (2024), 1-18.

\bibitem{C&L01} L. F. Cheung, P. F. Leung, \textit{Eigenvalue estimates for submanifolds with bounded mean curvature in the hyperbolic space}, Math. Z. {\bf 236} (2001), 525-530.

\bibitem{CL24} O. Chodosh, C. Li, \textit{Stable minimal hypersurfaces in $\mathbb{R}^4$}, Acta Math. {\bf 233} (2024), 1-31.

\bibitem{CLMS} O. Chodosh, C. Li, P. Minter, D. Stryker, \textit{Stable minimal hypersurfaces in $\mathbb{R}^5$}, to appear in Annals of Mathematics, arXiv:2401.01492.

\bibitem{G65} E. De Giorgi, \textit{Una estensione del teorema di Bernstein}, Ann. Scuola Norm. Sup. Pisa Cl. Sci. {\bf 19} (3) (1965), 79–85.

\bibitem{CD83} M. Do Carmo, M. Dajczer, \textit{Rotation hypersurfaces in spaces of constant curvature}, Trans. Amer. Math. Soc. {\bf 277} (1983), 685-709.

\bibitem{CP79} M. Do Carmo, C. K. Peng, \textit{Stable complete minimal surfaces in $\mathbb{R}^3$ are planes}, Bull. Amer. Math. Soc. {\bf 1} (1979), 903–906.

\bibitem{CP80} M. Do Carmo, C. K. Peng, \textit{Stable complete minimal hypersurfaces}, Proc. Beijing Sympos. Differential Equations Differential Geom. {\bf 3} (1980), 1349-1358.

\bibitem{D&S12} N. T. Dung, K. Seo, \textit{Stable minimal hypersurfaces in a Riemannian manifold with pinched negative sectional curvature}, Ann. Global Anal. Geom. {\bf 41} (2012), 447-460.

\bibitem{FC&S80} D. Fischer-Colbrie, R. Schoen, \textit{The structure of complete stable minimal surfaces in 3-manifolds of non-negative scalar curvature}, Comm. Pure Appl. Math. {\bf 33} (1980), 199-211.

\bibitem{F62} W. H. Fleming, \textit{On the oriented Plateau problem}, Rend. Circ. Mat. Palermo {\bf 11} (1962), 69–90.

\bibitem{G&T01} D. Gilbarg, N. S. Trudinger, \textit{Elliptic partial differential equations of second order}, Classics in Mathematics, Springer (2001). Reprint of the 1998 edition.

\bibitem{KN93} S. Kato, S. Nayatani, \textit{Complete conformal metrics with prescribed scalar curvature on subdomains of a compact manifold}, Nagoya Math. J. {\bf 132} (1993), 155-173.

\bibitem{KW} J. L. Kazdan, F. W. Warner, \textit{Scalar curvature and conformal deformation of Riemannian structure}, J. Differential Geometry {\bf 10} (1975), 113–134.

\bibitem{M} L. Mazet, \textit{Stable minimal hypersurfaces in $\mathbb{R}^6$}, preprint, arXiv:2405.14676.

\bibitem{M81}  H. Mori, \textit{Minimal surfaces of revolution in $\mathbb{H}^3$ and their global stability}, Indiana Univ. Math. J. {\bf 30} (1981), 787-794.

\bibitem{srg} B. O’Neill, \textit{Semi-riemannian geometry with application to relativity}, Pure Appl. Math. {\bf 103}, Academic Press (1983).

\bibitem{P16} P. Petersen, \textit{Riemannian geometry}, 3rd ed. Graduate Texts in Mathematics {\bf 171}, Springer (2016).

\bibitem{Schoen&75} R. Schoen, L. Simon, S-T. Yau, \textit{Curvature estimates for minimal hypersurfaces}, Acta Math. {\bf 134} (3-4) (1975), 275-288.

\bibitem{Seo11} K. Seo, \textit{Stable minimal hypersurfaces in the hyperbolic space}, J. Korean Math. Soc. {\bf 48} (2) (2011), 253-266.

\bibitem{Seo12} K. Seo, \textit{Eigenvalue estimates for stable minimal hypersurfaces}, Math. Inequal. Appl. {\bf 15} (1) (2012), 69-75.

\bibitem{SZ98} Y. B. Shen, X. H. Zhu, \textit{On stable complete minimal hypersurfaces in $\mathbb{R}^{n+1}$}, Amer. J. Math. {\bf 120} (1998), 103-116.

\bibitem{S68} J. Simons, \textit{Minimal varieties in riemannian manifolds}, Ann. of Math. {\bf 88} (2) (1968), 62–105.


\bibitem{Xin05} Y. L. Xin, \textit{Berstein type theorems without graphic condition}, Asian J. Math. {\bf 9} (1) (2005), 31-44.

\end{thebibliography}
\end{document}